\documentclass{amsart}
\usepackage[utf8]{inputenc}
\usepackage{amsmath}
\usepackage{amsthm}
\usepackage{amsfonts, dsfont}
\usepackage{amssymb}
\usepackage[all]{xy}
\usepackage{indentfirst}
\usepackage[pagebackref=true]{hyperref}

\usepackage{cleveref}
\usepackage[alphabetic, initials]{amsrefs}

\usepackage{faktor}
\usepackage{xfrac}

\newtheorem{thm}{Theorem}[section]

\newtheorem{theorem}[thm]{Theorem}
\newtheorem{corollary}[thm]{Corollary}
\newtheorem{proposition}[thm]{Proposition}

\newtheorem*{theorem*}{Theorem}

\theoremstyle{definition}

\newtheorem*{defn*}{Definiton}
\newtheorem{example}[thm]{Example}

\newtheorem*{ack}{Acknowledgement}
\newtheorem{remark}[thm]{Remark}

\newcommand{\N}{\mathbb{N}} 
\newcommand{\Z}{\mathbb{Z}} 

\usepackage{verbatim} 
\usepackage{color}
\newcommand{\G}{\Gamma}

\newcommand{\La}{\Lambda}

\newcommand{\act}{\!\curvearrowright\!}

\newcommand{\supp}{\operatorname{supp}}

\newcommand{\Prob}{\operatorname{Prob}}

\newcommand{\cC}{\mathcal{C}}

\newcommand{\cL}{\mathcal{L}}

\newcommand{\cW}{\mathcal{W}}

\title{Relatively amenable actions of Thompson's groups}
\author{Eduardo Scarparo}

\address{Eduardo Scarparo\\ Federal University of Santa Catarina\\ Brazil}
\email{eduardo.scarparo@posgrad.ufsc.br}

\begin{document}


\begin{abstract}
We investigate the notion of relatively amenable topological action and show that the action of Thompson's group $T$ on $S^1$ is relatively amenable with respect to Thompson's group $F$.  We use this to conclude that $F$ is exact if and only if $T$ is exact.  Moreover, we prove that the groupoid of germs of the action of $T$ on $S^1$ is Borel amenable. 
\end{abstract}
\maketitle
\section{Introduction}

In \cite{S07}, Spielberg showed that every \emph{Kirchberg} (i.e., simple, nuclear, purely infinite and seperable) \emph{algebra} which satisfies the UCT (universal coefficient theorem) admits a Hausdorff groupoid model, hence admits Cartan subalgebra.  Conversely, it was shown by Barlak and Li in \cite{BL17} that any separable and nuclear $C^*$-algebra which has a Cartan subalgebra satisfies the UCT.

Given an étale non-Hausdorff groupoid $G$,  there are dynamical criteria which ensure that the essential $C^*$-algebra of $G$ is a Kirchberg algebra.  Since, in general, $C^*_{\mathrm{ess}}(G)$ does not admit any obvious Cartan subalgebra, it seems natural to look at such groupoids as potential sources of counterexamples to the UCT problem (of whether every separable nuclear $C^*$-algebra satisfies the UCT).

Let $G(T,S^1)$ be the groupoid of germs of the action of Thompson's group $T$ on $S^1$. In \cite{KS20}, Kalantar and the author showed that the reduced $C^*$-algebra of $G(T,S^1)$ is not simple, even though $G(T,S^1)$ is minimal and effective. Moreover, as observed in \cite{KS20}, it follows from results of Kwaśniewski and Meyer \cite{KM21} that $C^*_{\mathrm{ess}}(G(T,S^1))$ is purely infinite and simple.  In this work, we show that $G(T,S^1)$ is Borel amenable. Since, as observed by Renault in \cite{Ren15}, the results on nuclearity of groupoid $C^*$-algebras from the work of Anantharaman-Delaroche and Renault \cite{zbMATH01544087} use only Borel amenability and hold in the non-Hausdorff setting as well, we conclude that $C^*_{\mathrm{ess}}(G(T,S^1))$ is a Kirchberg algebra. We leave open the question of whether $C^*_{\mathrm{ess}}(G(T,S^1))$ admits Cartan subalgebra (equivalently, whether it satisfies the UCT).

Let $\G$ be a group acting on a locally compact Hausdorff space $X$ and on a set $K$.  In \cite{zbMATH05053848}, Ozawa studied the existence of nets of continuous approximately equivariant maps $\mu_i\colon X\to\Prob(K)$. Clearly, the existence of such maps generalizes both topological amenability (in the case $K=\G)$ and set-theoretical amenability (in the case that $X$ consists of a single point).  If such a property holds in the case that $K$ is a set of left cosets $\G/\La$, we say that $X$ is \emph{$(\G,\La)$-amenable}. We show that this property generalizes the notion of \emph{relative co-amenability} introduced by Caprace and Monod in \cite{CM14} (in the more general setting of locally compact groups).

Consider Thompson's groups $F\leq T$. We show that $S^1$ is $(T,F)$-amenable and use this fact to conclude that $F$ is exact if and only if also $T$ is exact.

\begin{ack}
I am grateful to Nicolas Monod for comments on a preliminary version.
\end{ack}

\section{Relatively amenable actions}

Given a set $Y$, we consider $\Prob(Y):=\{\mu\in \ell^1(Y):\mu\geq 0, \|\mu\|_1=1\}$ equipped with the pointwise convergence topology.  

Given a group $\G$ acting by homeomorphisms on a locally compact Hausdorff space $X$, we say that $X$ is a \emph{locally compact $\G$-space}.  Given $\La\leq \G$, we say that $X$ is \emph{$(\G,\La)$-amenable} if there exists a net of continuous functions $\mu_i\colon X\to \Prob(\G/\La)$ which is \emph{approximately invariant} in the sense that 
\begin{equation*}
\lim_i\sup_{x\in K}\|s\mu_i(x)-\mu_i(sx)\|_1= 0
\end{equation*}
 for all $s\in\G$ and $K\subset X$ compact.  If $\La=\{e\}$, this is the usual notion of (\emph{topologically}) amenable action on a space $X$(\cite[Example 2.2.14.(2)]{zbMATH01544087}).  If $\La$ is co-amenable in $\G$, then any $\G$-space is $(\G,\La)$-amenable.

We will need the following result later on:
\begin{proposition}[{\cite[Proposition 5.2.1]{zbMATH05256855}}]\label{prop:BO}
Let $X$ be a compact $\G$-space which is $(\G,\La)$-amenable for some $\La\leq \G$.  If $\La$ is exact, then $\G$ is exact.
\end{proposition}

Let us now characterize $(\G,\La)$-amenability in the case of a discrete $\G$-space.
\begin{proposition}\label{prop:rc}
Let $S$ be a discrete $\G$-space and $\La\leq\G$. Then $S$ is $(\G,\La)$-amenable if and only if there exists a unital positive $\G$-equivariant linear map $\varphi\colon \ell^\infty(\G/\La)\to\ell^\infty(S)$.
\end{proposition}

\begin{proof}
We identify the space of bounded linear maps $\cL(\ell^\infty(\G/\La),\ell^\infty(S))$ with $\ell^\infty(S,\ell^\infty(\G/\La)^*)$. Under this identification, a unital positive $\G$-equivariant map $\varphi\in\cL(\ell^\infty(\G/\La),\ell^\infty(S))$ corresponds to a map $\psi\colon S\to\ell^\infty(\G/\La)^*$ such that, for every $s\in S$ and $g\in\G$,  it holds that $\psi(s)$ is a state and $\psi(gs)=g(\psi(s))$.

Suppose $S$ is $(\G,\La)$-amenable and let $\mu_i\colon S\to\Prob(\G/\La)\subset\ell^\infty(\G/\La)^*$ be a net of approximately invariant functions.  By taking a subnet, we may assume that, for each $s\in S$, $\mu_i(s)$ converges in the weak-$*$ topology to a state $\psi(s)\in\ell^\infty(\G/\La)^*$. Clearly,  $\psi\colon S\to\ell^\infty(\G/\La)^*$ has the desired properties.

Conversely, suppose that there exists $\psi\in\ell^\infty(S,\ell^\infty(\G/\La)^*)$ unital, positive and $\G$-equivariant.  Since $\ell^1(\G)$ is weak-$*$ dense in $\ell^\infty(\G)^*$,  we can find a net $\mu_i\colon S\to\Prob(\G/\La)\subset\ell^\infty(\G/\La)^*$ such that, for each $s\in S$, $\mu_i(s)\to\psi(s)$ in the weak-$*$ topology.  By $\G$-equivariance of $\psi$, we have that, for each $g\in\G$ and $s\in S$, the net $g\mu_i(s)-\mu_i(gs)$ converges to zero weakly in $\ell^1(\G/\La)$. 

Given $\epsilon>0$ and finite subsets $E\subset \G$ and $F\subset S$, we claim that there is $\mu\colon S\to\Prob(\G/\La)$ such that $\|g\mu(x)-\mu(gx)\|_1<\epsilon$ for each $x\in F$ and $g\in E$.  From the previous paragraph, we obtain that $0$ is in the weak closure of the convex set 
$$\bigoplus_{\substack{g\in E\\s\in F}}\{g\mu(s)-\mu(gs)\mid\mu\colon S\to\Prob(\G/\La)\}\subset\bigoplus_{\substack{g\in E\\s\in F}}\ell^1(\G/\La)$$.

By the Hahn-Banach separation theorem, the claim follows. Thus, $S$ is $(\G,\La)$-amenable.
\end{proof}

\begin{remark}
Given a group $\G$ and subgroups $\La_1,\La_2\leq\G$, Proposition \ref{prop:rc} implies that $\G/\La_2$ is $(\G,\La_1)$-amenable if and only if $\La_1$ is co-amenable to $\La_2$ relative to $\G$ in the sense of \cite[Section 7.C]{CM14}.

\end{remark}

For completeness, we record the following permanence property. The proof follows the argument in \cite[Proposition 5.2.1]{zbMATH05256855}.

\begin{proposition}\label{ext}
Let $X$ be a locally compact $\G$-space and $\La_1\leq\La_2\leq\G$ such that $X$ is $(\G,\La_2)$-amenable and $(\La_2,\La_1)$-amenable. Then $X$ is $(\G,\La_1)$-amenable.
\end{proposition}
\begin{proof}
Fix $E\subset \G$ finite, $\epsilon>0$ and $K\subset X$ compact.  Take $\eta\colon X\to\Prob(\G/\La_2)$ continuous such that $\sup_{x\in K}\|s\eta^x-\eta^{sx}\|<\epsilon/2$ for all $s\in E$. By arguing as in \cite[Lemma 4.3.8]{zbMATH05256855}, we may assume that there is $F\subset\G/\La_2$ finite such that $\supp\eta^x\subset F$ for all $x\in X$.

Fix a cross-section $\sigma\colon\G/\La_2\to\G$. Let $$E^*:=\{\sigma(sa)^{-1}s\sigma(a
):a\in F, s\in E\}\subset\La_2.$$

and $$L:=\bigcup_{a\in F}\sigma(a)^{-1}K.$$

 Take $\nu\colon X\to\Prob(\La_2/\La_1)\subset\Prob(\G/\La_1)$ continuous such that $$\max_{s\in E^*}\sup_{y\in L}\|s\nu(y)-\nu(sy)\|_1<\epsilon/2.$$

Let 
\begin{align*}
\mu\colon X&\to \Prob(\G/\La_1)\\
x&\mapsto \sum_{a\in F}\eta^x(a)\sigma(a)\nu^{\sigma(a)^{-1}x}.
\end{align*}

Given $s\in E$ and $x\in K$, we have 
\begin{align*}
s\mu(x)&=\sum_{a\in F}\eta^x(a)s\sigma(a)\nu^{\sigma(a)^{-1}x}\\
&=\sum_{a\in \G/\La_2}\eta^x(a)\sigma(sa)\sigma(sa)^{-1}s\sigma(a)\nu^{\sigma(a)^{-1}x}\\
&\approx_{\epsilon/2} \sum_{a\in \G/\La_2}\eta^x(a)\sigma(sa)\nu^{\sigma(sa)^{-1}sx}\\
&\approx_{\epsilon/2} \sum_{a\in \G/\La_2}\eta^{sx}(sa)\sigma(sa)\nu^{\sigma(sa)^{-1}sx}\\
&=\sum_{b\in \G/\La_2}\eta^{sx}(b)\sigma(b)\nu^{\sigma(b)^{-1}sx}=\mu(sx).
\end{align*}
\end{proof}

\subsection*{Thompson's groups} Thompson's group $V$ consists of piecewise linear, right continuous bijections on $[0,1)$ which have finitely many points of non-differentiability, all being dyadic rationals, and have a derivative which is an integer power of $2$ at each point of differentiability.  

Let $\cW$ be the set of finite words in the alphabet $\{0,1\}$. Given $w\in \cW$ with length $|w|$, let $\cC(w):=\{(x_n)\in \{0,1\}^\N:x_{[1,|w|]}=w\}$.  Also let $\psi\colon\cW\to[0,1]$ be the map given by $\psi(w):=\sum_{n=1}^{|w|} x_n2^{-n}$, for $w\in \cW$.  By identifying a set of the form $\cC(w)$ with the half-open interval $[\psi(w),\psi(w)+2^{-|w|})$, we can view $V$ as the group of homeomorphisms of $\{0,1\}^\N$ consisting of elements $g$ for which there exist two partitions $\{\cC(w_1),\dots\cC(w_n)\}$ and $\{\cC_{z_1},\dots,\cC_{z_n}\}$ of $\{0,1\}^\N$ such that $g(w_ix)=z_ix$ for every $i$ and infnite binary sequence $x$. 

Let $D:=\{(x_n)\in \{0,1\}^\N:\text{$\exists k\in\N$ such that $x_l=0\ \ \forall l\geq k$}\}$. Notice that $D$ is $V$-invariant. Given $w\in \cW$, let $w0^\infty$ be the element of $D$ obtained by extending $w$ with infinitely many $0$'s.

\begin{theorem}\label{thompson}
There is a sequence of continuous maps $\mu_N\colon \{0,1\}^\N\to\Prob(D)$ such that 
\begin{equation}\label{limi}
\lim_N\sup_{x\in \{0,1\}^\N}\|s\mu_N(x)-\mu_N(sx)\|_1=0
\end{equation}
for every $s\in V$.
\end{theorem}

\begin{proof}

Given $N\in \N$, let $\mu_N\colon \{0,1\}^\N\to\Prob(D)$ be defined by $$\mu_N(x):=\frac{1}{N}\sum_{j=1}^N\delta_{x_{[1,j]}0^\infty}.$$  Clearly,  for each $d\in D$ and $N\in\N$, the map $x\mapsto \mu_N(x)(d)$ is continuous. We claim that $(\mu_N)$ satisfies $\eqref{limi}$. 

Fix $s\in V$. There exist two partitions $\{\cC(w_1),\dots\cC(w_n)\}$ and $\{\cC_{z_1},\dots,\cC_{z_n}\}$ of $\{0,1\}^\N$ such that $s(w_ix)=z_ix$ for every $i$ and infnite binary sequence $x$. 

Let $k(s):=\max_i\{|w_i|,|z_i|-|w_i|\}$.  Fix $1\leq i \leq n$ and $x\in \cC(w_i)$.  Let $\alpha_i:=|z_i|-|w_i|$. Given $k> k(s)$, we have 
$$s(x_{[1,k]}0^\infty)
=z_ix_{[|w_i|+1,k]}0^\infty
=s(x)_{[1,|z_i|]}s(x)_{[|z_i|+1,k+|z_i|-|w_i|]}0^\infty
=s(x)_{[1,k+\alpha_i]}0^\infty.$$  
For $N>2k(s)$, we have 
\begin{align*}
\|s\mu_N(x)-\mu_N(sx)\|&=\frac{1}{N}\left\|\sum_{j=1}^N\delta_{s(x_{[1,j]}0^\infty)}-\delta_{s(x)_{[1,j]}0^\infty}\right\|\\
&\leq \frac{1}{N}\left\|\sum_{j=k(s)+1}^{N-k(s)}\delta_{s(x_{[1,j]}0^\infty)}-\sum_{l=k(s)+1+\alpha_i}^{N-k(s)+\alpha_i}\delta_{s(x)_{[1,l]}0^\infty}\right\| + \frac{4k(s)}{N}  \\
&=\frac{4k(s)}{N}.
\end{align*}

\end{proof}

Thompson's group $T$ is the subgroup of $V$ consisting of elements which have at most one point of discontinuity.  By identifying $[0,1)$ with $S^1$, the elements of $T$ can be seen as homeomorphisms on $S^1$. Thompson's group $F$ is the subgroup of $T$ which stabilizes $1\in S^1$. 

\begin{corollary}\label{cor:amena}
The spaces $\{0,1\}^\N$ and $S^1$ are $(T,F)$-amenable.
\end{corollary}
\begin{proof}
Notice that $T$ acts transitively on $D\subset \{0,1\}^\N$. Since $F$ is the stabilizer of $0^\infty\in D$, it follows immediately from Theorem \ref{thompson} that $\{0,1\}^\N$ is $(T,F)$-amenable.

Let $\varphi\colon S^1\to \{0,1\}^\N$ be the map which, given $\theta\in[0,1)$,  sends $e^{2\pi i \theta}$ to the binary expansion of $\theta$. Clearly, $\varphi$ is $T$-equivariant and Borel measurable.  Since $\{0,1\}^\N$ is $(T,F)$-amenable, composition with $\varphi$ gives rise to a sequence $u_n\colon S^1\to\Prob(T/F)$ of approximately $T$-equivariant pointwise Borel maps (in the sense that for each $d\in T/F$ the map $x\mapsto u_n(x)(d)$ is Borel).  It follows from \cite[Proposition 5.2.1]{zbMATH05256855} (or \cite[Proposition 11]{zbMATH05053848}) that $S^1$ is $(T,F)$-amenable.
\end{proof}

The next result follows immediately from Proposition \ref{prop:BO} and Corollary \ref{cor:amena}.

\begin{corollary}
Thompson's group $F$ is exact if and only if Thompson's group $T$ is exact.
\end{corollary}

The next result has been recorded in \cite[Section 3.2]{Mon06} as a consequence of hyperfiniteness of the equivalence relation of $T$ on $S^1$. It also follows from the fact that stabilizers of amenable actions are amenable, Proposition \ref{ext} and Corollary \ref{cor:amena}.

\begin{corollary}[\cite{Mon06}]
The following conditions are equivalent:
\begin{enumerate}
\item[(i)] $F$ is amenable;
\item[(ii)] $T\act \{0,1\}^\N$ is amenable;
\item[(iii)] $T\act S^1$ is amenable.

\end{enumerate}
\end{corollary}

\section{Groupoids of germs}
We say that a topological groupoid $G$ is \emph{étale} if its unit space $G^{(0)}$ is Hausdorff and the range and source maps $r,s\colon G\to G^{(0)}$ are local homeomorphisms. If $G$ is also second countable,  then $G$ is said to be \emph{Borel amenable} (\cite[Definition 2.1]{Ren15}) if there exists a sequence $(m_n)_{n\in\N}$, where each $m_n$ is a family $(m_n^x)_{x\in G^{(0)}}$ of probability measures on $r^{-1}(x)$ such that:
\begin{enumerate}
\item[(i)] For all $n\in\N$,  $m_n$ is Borel in the sense that for all bounded Borel functions $f$ on $G$, $x\mapsto\sum_{g\in r^{-1}(x)}f(g)m_n^x(g)$ is Borel;
\item[(ii)] For all $g\in G$,  we have $\sum_{h\in r^{-1}(r(g))}|m_n^{s(g)}(g^{-1}h)-m_n^{r(g)}(h)|\to 0$.
\end{enumerate}

\begin{remark}\label{subgpd}
Let $G$ be a second countable étale groupoid and $A\subset G^{(0)}$ a measurable subset which is invariant in the sense that $r^{-1}(A)=s^{-1}(A)$. In this case, $G_A:=s^{-1}(A)$ is a subgroupoid of $G$.  If $G$ is Borel amenable, then clearly $G_A$ is also Borel amenable.  Conversely,  if $G_A$ and $G_{G^{(0)}\setminus A}$ are Borel amenable, then, since $G=G_A\sqcup G_{G^{(0)}\setminus A}$, also $G$ is Borel amenable.
\end{remark}

Let $\G$ be a group acting on a compact Hausdorff space $X$. Given $x\in X$, let $\G_x^{0}:=\{g\in\G:\text{$g$ fixes pointwise a neighborhood of $x$}\}$ be the \emph{open stabilizer} at $x$. Consider the following equivalence relation on $\G\times X$: $(g,x)\sim(h,y)$ if and only if $x=y$ and $g\G_x^{(0)}=h\G_x^{(0)}$. As a set, the \emph{groupoid of germs} of $\G\act X$ is $G(\G,X):=\frac{\G\times X}{\sim}$. The topology on $G(\G,X)$ is the one generated by sets of the form $[g,U]:=\{[g,x]:x\in U\}$, for $U\subset X$ open and $g\in \G$.  Inversion in $G(\G,X)$ is given by $[g,x]^{-1}=[g^{-1},gx]$. Two elements $[h,y],[g,x]\in G(\G,X)$ are multipliable if and only if $y=gx$, in which case $[h,y][g,x]:=[hg,x]$. With this structure, $G(\G,X)$ is an étale groupoid.

\begin{example}\label{ex:tg}
Let $G_{[2]}$ be the Cuntz groupoid introduced in \cite[Definition III.2.1]{Ren80}. Since Thompson's group $T$ can be seen as a covering subgroup of the topological full group of $G_{[2]}$ (\cite[Example 3.3]{BS19}), it follows from \cite[Proposition 4.10]{NO19} that $G(T,\{0,1\}^\N)\simeq G_{[2]}$. Hence, $G(T,\{0,1\}^\N)$ is Borel amenable by \cite[Proposition III.2.5]{Ren80}.
\end{example}
\begin{theorem}
The groupoid of germs of $T\act S^1$ is Borel amenable.
\end{theorem}
\begin{proof}
Let $X:=\{e^{2\pi i\theta}:\theta\in\Z[1/2]\}$ and $Y:=S^1\setminus X$.  Notice that $X$ is $T$-invariant. We will show that $G(T,S^1)_X$ and $G(T,S^1)_Y$ are Borel amenable. From Remark \ref{subgpd} it will follow that $G(T,S^1)$ is Borel amenable.

Let $\varphi\colon S^1\to \{0,1\}^\N$ be the $T$-equivariant map which, given $\theta\in[0,1)$,  sends $e^{2\pi i \theta}$ to the binary expansion of $\theta$. Notice that $\varphi|_Y\colon Y\to \varphi(Y)$ is a homeomorphism. Furthermore, the map 
\begin{align*}
\tilde{\varphi}\colon G(T,S^1)_Y&\to G(T,\{0,1\}^\N)_{\varphi(Y)}\\
[g,y]&\mapsto[g,\varphi(y)]
\end{align*}
 is an isomorphism of topological groupoids. Therefore, $G(T,S^1)_Y$ is Borel amenable by Remark \ref{subgpd} and Example \ref{ex:tg}.

Notice that $G(T,S^1)_X$ is a countable set. Moreover, it follows from \cite[Theorem 4.1]{CFP96} that the open stabilizer $T_1^0$ is equal to the commutator subgroup $[F,F]$ and $\frac{F}{[F,F]}\simeq\Z^2$.  Therefore, $G(T,S^1)_X$ is Borel isomorphic to the transitive discrete groupoid $X\times X\times\Z^2$, which, due to the amenability of the isotropy group, is Borel amenable.

\end{proof}

\bibliographystyle{acm}
\bibliography{bibliografia}
\end{document}